\documentclass[11pt,twoside, leqno]{article}

\usepackage{graphics}
\usepackage{epsfig}
\usepackage{color}
\usepackage{verbatim}
\usepackage{citeref}
\usepackage{hyperref}
\usepackage[displaymath,mathlines,pagewise]{lineno}
\usepackage{amssymb,amsmath,amsfonts,amsthm,color,mathrsfs}
\usepackage[Symbol]{upgreek}
\usepackage{txfonts}
\usepackage[nottoc,notlot,notlof]{tocbibind}
\usepackage[active]{srcltx}

\allowdisplaybreaks
\def\rr{{\mathbb R}}
\def\rn{{{\rr}^n}}

\def\cn{{\mathbb N}}

\def\fz{\infty}
\def\az{\alpha}
\def\supp{{\mathop\mathrm{\,supp\,}}}
\def\dist{{\mathop\mathrm {\,dist\,}}}
\def\diam{{\mathop\mathrm {\,diam\,}}}
\def\loc{{\mathop\mathrm{\,loc\,}}}

\def\ez{\epsilon}

\def\ro{\rho}

\def\gz{{\gamma}}

\def\r{\right}
\def\lf{\left}

\def\N{\mathbb{N}}

\newcommand{\abs}[1]{\left\vert #1 \right\vert}

\newtheorem{thm}{Theorem}[section]
\newtheorem{lem}[thm]{Lemma}
\newtheorem{prop}[thm]{Proposition}
\newtheorem{cor}[thm]{Corollary}

\newtheorem{quest}[thm]{Question}

\newtheorem{defn}[thm]{Definition}

\newtheorem{exa}[thm]{Example}
\numberwithin{equation}{section}

\begin{document}
\arraycolsep=1pt
\author{Renjin Jiang and Aapo Kauranen } \arraycolsep=1pt
\title{\Large\bf Korn's inequality and John domains}
 \footnotetext{
\endgraf  }
\date{\today }
\maketitle
\begin{center}
\begin{minipage}{11cm}\small
{\noindent{\bf Abstract:} Let $\Omega \subset \mathbb{R}^n$, $n\ge2$, be a bounded
 domain satisfying the separation property. We show that the
following conditions are equivalent:

(i) $\Omega$ is a John domain;

(ii) for a fixed $p\in (1,\infty)$, the Korn inequality holds for each $\mathbf{u}\in
W^{1,p}(\Omega,\mathbb{R}^n)$ satisfying $\int_\Omega \frac{\partial
u_i}{\partial x_j}-\frac{\partial u_j}{\partial x_i}\,dx=0$, $1\le i,j\le n$,
$$\|D\mathbf{u}\|_{L^p(\Omega)}\le C_K(\Omega, p)\|\epsilon(\mathbf{u})\|_{L^p(\Omega)}; \leqno(K_p)$$

(ii') for all $p\in (1,\infty)$, $(K_p)$ holds on $\Omega$;

(iii)   for a fixed $p\in (1,\infty)$,  for each $f\in L^p(\Omega)$ with vanishing
mean value on $\Omega$,
there exists a solution $\mathbf{v}\in W^{1,p}_0(\Omega,\mathbb{R}^n)$ to the equation $\mathrm{div}\,\mathbf{v}=f$ with
$$\|\mathbf{v}\|_{W^{1,p}(\Omega,\mathbb{R}^n)}\le C(\Omega, p)\|f\|_{L^p(\Omega)}; \leqno(DE_p)$$

(iii') for all $p\in (1,\infty)$, $(DE_p)$ holds on $\Omega$.

For domains satisfying the separation property, in particular, for  finitely
connected domains in the plane,
our result provides a geometric characterization of the Korn inequality, and gives positive answers to
a question raised by [Costabel \& Dauge, Arch. Ration. Mech. Anal. 217 (2015), 873-898]
and a question raised by [Russ, Vietnam J. Math. 41 (2013), 369-381]. For the plane, our result is best possible
in the sense that, there exist infinitely connected domains which are not John
but support Korn's inequality.
 }\end{minipage}
\end{center}
\vspace{1.0cm}

\section{Introduction}
\hskip\parindent Let $\Omega$ be a bounded domain in $\rr^n$, $n\ge 2$. Throughout the paper,
 a domain is a connected, open and proper subset of $\rr^n$, $n\ge 2$.
 Let $p\in (1,\infty)$.  For each
vector field $\mathbf{v}=(v_1,\cdots,v_n)\in W^{1,p}(\Omega,\rr^n)$,  let
$D\mathbf{v}$ denotes its (distributional) gradient matrix, $\epsilon
(\mathbf{v})$ denotes the symmetric part of $D\mathbf{v}$., i.e.,
$\epsilon(\mathbf{v})=(\epsilon_{i,j}(\mathbf{v}))_{1\le i,j\le n}$
with $$\epsilon_{i,j}(\mathbf{v})=\frac 12\left(\frac{\partial
v_i}{\partial x_j}+\frac{\partial v_j}{\partial x_i}\right), $$
and $\kappa(\mathbf{v})=\{\kappa_{i,j}(\mathbf{v})\}_{1\le i,j\le n}$ be the anti-symmetric part of  $D\mathbf{v}$ as $\frac 12(D\mathbf{v}-D\mathbf{v}^T)$.

Korn's (second) inequality states that, if each $\kappa_{i,j}(\mathbf{v})$ is of vanishing mean value on $\Omega$,  then there exists $C_K>0$ such that
$$\|D\mathbf{v}\|_{L^p(\Omega)}\le C_K\|\ez(\mathbf{v})\|_{L^p(\Omega)}.\leqno(K_{p})$$
The Korn inequality $(K_p)$ is a fundamental tool in the theory of  elasticity
 and fluid mechanics; we refer the reader to
\cite{adm06,ci14,cd15,dl76,f47,ho95,ko89} and the references therein.

The study of Korn inequality on domains has a long history. Friedrichs \cite{f47} proved the Korn inequality $(K_p)$ for $p=2$ on
domains with a finite number of corners or edges on $\partial \Omega$. Ne$\check{\rm c}$as \cite{ne67} proved the Korn inequality on bounded
Lipschitz domains.
Kondratiev and Oleinik \cite{ko89} studied the Korn inequality $(K_2)$ on star-shaped domains.
Recently, Acosta, Dur\'an and Muschietti \cite{adm06} proved the Korn inequality (in a different form than $(K_p)$)
holds for all $p\in (1,\fz)$ on John domains (see Definition \ref{def-john} below).
We refer the reader to
\cite{adl06,adf13,cc05,ci14,drs10,jk15,mm71,nit81,ti72} for more studies on the Korn inequality. It is
worth to mention that  the Korn inequality $(K_p)$ fails for $p=1$ even on a cube; see the examples from \cite{cfm05,or62}.

Recently, there have been some studies concerning the Korn inequality for more irregular domains, such as H\"older domains and $s$-John domains with $s>1$.
 It turns out the Korn inequality $(K_p)$ does not hold for any $1<p<\infty$ on
these domains, and instead there are weighted versions of the Korn inequality;
see \cite{adl06,adf13,jk15} for instance.

 Then one may wonder, is John condition also necessary for domains to  support
the Korn inequality? Generally speaking, one may ask the following question.

\begin{quest}\label{qa}
 What is the geometric counterpart of the Korn inequality $(K_p)$? Is there a necessary and sufficient geometric condition on the domain under consideration for the Korn inequality to hold?
\end{quest}

On the other hand, there are some inequalities that are known to be equivalent to the Korn inequality, on (regular)  planar domains;
cf. \cite{cd15,hp83,ti01}. In \cite{hp83}, Horgan and Payne proved
that on {\em  simply connected regular} domains in the plane,
the Korn inequality $(K_2)$, the Friedrichs inequality and  the Babu\v ska-Aziz inequality
are equivalent to each other with equivalence of constants.

In what follows, for each $p\in (1,\infty)$ the space $L^p_0(\Omega)$ denotes the collection of all $L^p$-integrable functions on $\Omega$ with vanishing mean value. Recall that, the Friedrichs inequality states that for all conjugate harmonic functions $h,g$ on a plane domain $\Omega$,
if $h,g\in L^2_0(\Omega)$, then it holds
$$\|h\|_{L^2(\Omega)}\le C\|g\|_{L^2{(\Omega)}};$$
cf. \cite{cd15,f37,hp83}.
The  Babu\v ska-Aziz inequality  (cf. \cite{ba72,cd15,hp83}) for
the divergence equation states that,
for each $f\in L^2_0(\Omega)$
there exists a solution $\mathbf{v}\in W^{1,p}_0(\Omega,\rr^2)$ to the equation $\mathrm{div}\,\mathbf{v}=f$ with
$$\|\mathbf{v}\|_{W^{1,2}(\Omega,\rr^2)}\le C\|f\|_{L^2(\Omega)}.$$

Recently, Costabel and Dauge  \cite{cd15} verified the equivalence of
the Friedrichs inequality  and  the  Babu\v ska-Aziz inequality with equality of constants on bounded domains,
without requiring any (smooth) regularity of the domain, and they asked the following question:

\begin{quest}\label{qb} Is Korn's inequality equivalent to the  Babu\v ska-Aziz inequality, without requiring
a priori (smooth) regularity on the domain under consideration?
\end{quest}
The implication of Korn's inequality from  the  Babu\v ska-Aziz inequality is
well-known (cf. \cite{adm06,cd15}),
while the reverse part remains open.

Ne$\check{\rm c}$as \cite{ne67} had proved $L^p$-version of the  Babu\v ska-Aziz inequality for $1<p<\infty$ on Lipschitz domains.
Acosta, Dur\'an and Muschietti \cite{adm06}  extended solvability of the divergence equation
on John domains in $\rn$, $n\ge 2$, via a constructive approach.
{ By developing a decomposition technique for John domains,
Diening, Ru\v zi\v cka and Schumacher \cite{drs10} further obtained
solvability of the divergence equation in weighted cases, and also weighted Korn's inequality.  }
Russ \cite{ru13}  asked if one can characterize domains for which
the divergence equation with Dirichlet boundary condition is solvable.
In \cite{jkk14}, the authors with P. Koskela showed that on a bounded simply connected domain
$\Omega$ in $\rr^2$, the divergence equation is solvable in $W^{1,p}_0(\Omega,\rr^2)$ ($1<p<\infty$), if and only if $\Omega$ is a John domain.
This implies that, in particular, on a bounded simply connected plane domain, the  Babu\v ska-Aziz inequality holds if and only if
the domain is John.

Therefore, if Korn's inequality is equivalent to the  Babu\v ska-Aziz inequality, then the domain in some sense has to be a John domain, and this
will also answer our Question A. The main aim of this paper is to provide
positive answers to Questions A and B, if the domain additionally satisfies a
condition known as the  separation property. This condition, which is defined
below, is valid on every finitely connected plane domain.

Before moving further, let us first recall the definition of a John domain. This terminology
was introduced  in \cite{ms79}, but these domains were studied already by F.
John \cite{j61}. In what follows, for each $x\in \Omega$,
we denote by $\rho(x)$ the distance of $x$ to the boundary of $\Omega$, i.e., $d(x,\partial \Omega)$.

\begin{defn}[John domain]\label{def-john}
A bounded domain $\Omega\subset \rn$ with a distinguished point
$x_0\in\Omega$ is a John domain if  there exists a constant $C_J>0$ such that
for all $x\in\Omega$, there is a curve $\gz: [0,l(\gz)]\to\Omega$
parametrised by arclength such that $\gz(0)=x$, $\gz(l(\gz))= x_0$, and
$\rho(\gamma(t))\ge C_Jt$.
\end{defn}
The definition says that each point in a John domain can be joined to the center point with a twisted cone. Clearly, the requirements for being a John domain are much weaker than that for being Lipschitz domains. In particular, every Lipschitz domain is a John domain and John domains may have fractal boundary (e.g., Von Koch snowflake domain).

An additional requirement on the domain, which we will rely on heavily, is the following separation property,
which was introduced by Buckley and Koskela \cite{bk95}, where the separation property was used
to prove the geometric counterpart of Sobolev-Poincar\'e inequalities.

\begin{defn}[Separation property]\label{separation property}
A domain $\Omega \subset\rn$ with a distinguished point
$x_0$ has the separation property if there is a constant $C_s\ge 1$ such that the following
holds: for each $x\in\Omega$ there is a curve $\gz:[0, 1]\to \Omega$ with $\gz(0)= x$,
$\gz(1)= x_0$, and such that for each $t$ either  $\gz([0, t])\subset B:= B(\gz(t), C_s\ro(\gz(t)))$ or each $y\in \gz([0, t]) \setminus B$ belongs to a different component of $\Omega\setminus \partial B$ than $x_0$. In the latter case,
 we call $B$ a separating ball, and call the union of components of $\Omega\setminus  B$ not containing $x_0$
 as \emph{$B$-end}  and denoted by $E_B.$
\end{defn}
The domains satisfying this separation property cannot have a lot of boundary
components, and they may have long tentacles but one cannot have flat cusps (see Section 6).

Notice that if a domain $\Omega$ has the separation property for some $C_s\ge 1$, then it has the separation property for any constant larger
than $C_s$. It follows from the definitions that every John domain has the separation property.
 More generally, it is proved in \cite{bk95}  that, every domain $\Omega$, which
is quasiconformally equivalent to a uniform domain (cf. Jones \cite{jo81}
or Section \ref{discussion}) has
the separation property. In particular, any finitely connected plane domain has
the
separation property; see \cite[p. 583]{bk95} and Corollary \ref{fint-sp}.

Our main result below gives positive answers to Question \ref{qa}, Question \ref{qb}  and \cite[Question 3]{ru13}, if the domain additionally satisfies the separation property.

\begin{thm}\label{main}
Let $\Omega$ be a bounded domain of $\rr^n$, $n\ge 2$. Suppose that $\Omega$
satisfies the separation property.  Then
the following statements are equivalent:

(i) $\Omega$ is a John domain;

(ii) for a fixed $p\in (1,\infty)$, the Korn inequality holds for each $\mathbf{u}\in W^{1,p}(\Omega,\rr^n)$
satisfying
$\int_\Omega\kappa_{i,j}(\mathbf{u})\,dx=0$, $1\le i,j\le n$, as
$$\|D\mathbf{u}\|_{L^p(\Omega)}\le C_K\|\epsilon(\mathbf{u})\|_{L^p(\Omega)};\leqno (K_{p})$$

(ii') the Korn inequality $(K_{p})$ holds for each $p\in (1,\infty)$;

(iii) for a fixed $p\in (1,\infty)$,  for
all $\mathbf{u}\in W^{1,p}(\Omega,\rr^n)$ it holds that
$$\lf\|D\mathbf{u}\r\|_{L^p(\Omega)}\le C_K\lf\{\|\epsilon(\mathbf{u})\|_{L^p(\Omega)}+\lf\|D\mathbf{u}\r\|_{L^p(Q)}\r\},
\leqno (\widehat{K}_{p})$$
where  $Q$ is a fixed cube compactly supported in $\Omega$;

(iii') the Korn inequality $(\widehat{K}_{p})$ holds for each $p\in (1,\infty)$;

(iv)  for a fixed $p\in (1,\infty)$,  for each $f\in L^p_0(\Omega)$,
there exists a solution $\mathbf{v}\in W^{1,p}_0(\Omega,\rn)$ to the equation $\mathrm{div}\,\mathbf{v}=f$ with
$$\|\mathbf{v}\|_{W^{1,p}(\Omega,\rn)}\le C_d\|f\|_{L^p(\Omega)};\leqno(DE_p)$$

(iv') $(DE_p)$ holds for each $p\in (1,\infty)$.
\end{thm}

We want to stress that the statement of Theorem \ref{main} does {\em not} hold
without some additional assumption on the domain; see Example \ref{infinite-nonjohn} and Example \ref{simply-nonjohn} below.

Notice that  $(K_p)$ is equivalent to the following Korn inequality for all $\mathbf{u}\in W^{1,p}(\Omega,\rn)$,
$$\inf_{S=-S^T}\|D\mathbf{u}-S\|_{L^p(\Omega)}\le C_K\|\epsilon(\mathbf{u})\|_{L^p(\Omega)}.\leqno(\widetilde K_p)$$
The implication $(K_p)\Rightarrow (\widetilde K_p)$  follows by applying $(K_p)$ to $\mathbf{u}-Sx$ for each $\mathbf{u}\in W^{1,p}(\Omega,\rn)$, where the elements $s_{i,j}$ of $S$ are chosen as $\int_{\Omega}\kappa_{i,j}(\mathbf{u})\,dx$;
the converse implication follows by noticing that for $\mathbf{u}\in W^{1,p}(\Omega,\rn)$ with $\int_\Omega \kappa_{i,j}(\mathbf{u})\,dx=0$,
$\|\kappa_{i,j}(\mathbf{u})\|_{L^p(\Omega)}\le C(\Omega,p)\|\kappa_{i,j}(\mathbf{u})-a\|_{L^p(\Omega)}$ for any $a\in \rr$.

By \cite[Theorem 1.1]{bk95}, any of the above seven conditions is further equivalent to a Sobolev-Poincar\'e inequality, as
$$\left(\int_\Omega |u-u_\Omega|^{{np}/({n-p})}\,dx\right)^{({n-p})/{np}}\le C\left(\int_\Omega |\nabla u|^{p}\,dx\right)^{1/p},$$
for some $p\in [1,n)$, where $u_\Omega$ denotes the average of integral of $u$ on $\Omega$.

The implication $(i)\Longrightarrow(iv)$ was proved in \cite{adm06}, where the separation property is not required.
In \cite{jkk14}, with Koskela, the authors have proven $(iv)\Longrightarrow (i),$  except for $p= n \geq 2,$ directly by using Poincar\'e inequality and Hardy inequality. Therefore, Theorem \ref{main} improves also this result slightly.

The implication $(vi)\Longrightarrow(ii), (iii)$ is well-known, we include a proof for completeness in Section 2. The main achievement of
this paper is that we show that  $(ii), (iii)\Longrightarrow (i)$ is also true,
if the domain satisfies the separation property,
which holds in particular on finitely connected domain $\Omega$ in the plane by Corollary \ref{fint-sp} below.

\begin{cor}\label{main-cor}
Let $\Omega$ be a bounded finitely connected domain in $\rr^2$.
Then the following statements are equivalent:

(i) $\Omega$ is a John domain;

(ii)  the Korn inequality $(K_p)$ holds for some $p\in(1,\infty)$;

(ii') the Korn inequality $(K_p)$ holds for each $p\in(1,\infty)$;


(iii)   the Korn inequality $(\widehat K_p)$ holds for some $p\in(1,\infty)$;

(iii')   the Korn inequality $(\widehat K_p)$ holds for each $p\in(1,\infty)$;

(iv)    $(DE_p)$ holds for some $p\in(1,\infty)$;

(iv') $(DE_p)$ holds for each $p\in(1,\infty)$.
\end{cor}

The result in dimension two is best possible in the sense that, our characterization of domains for the Korn inequality
works for any finitely connected domains, meanwhile,  one has

\begin{exa}\label{infinite-nonjohn}
There exist infinitely connected plane domains $\Omega$, which are not John
domains, but  the Korn inequalities $(K_p)$ and $(\widehat{K}_p)$
hold on $\Omega$ for all $p\in (1,\infty)$.
\end{exa}

We remark that one cannot hope for exact analogues of  Corollary \ref{main-cor} in dimensions $n\ge 3$,
since the connectivity property is a much weaker condition in higher dimensions than it is in dimension two,
and $\mathbb{R}^n$ is highly rigid when $n\ge 3$. This phenomenon
happens in many areas, for instance, the problem for Sobolev extension domains,
the Schoenfliess theorem fails in $\mathbb{R}^3$ and
there are very few (quasi)conformal mappings  in dimensions $n\ge 3$;  see Jones \cite{jo81}.
Indeed, the following example  shows even simple connectivity is not sufficient to deduce analogues of Corollary \ref{main-cor} in higher dimensions.
\begin{exa}\label{simply-nonjohn}
There exist simply connected domains $\Omega\subset \rr^n$, $n\ge 3$, which are not John domains, but  the Korn inequalities $(K_p)$ and $(\widehat{K}_p)$ hold on $\Omega$ for all $p\in (1,\infty)$.
\end{exa}
The above two examples are based on domains constructed in Buckley-Koskela \cite{bk95}, details will be given in Section \ref{discussion}.

 Corollary \ref{main-cor} together with \cite[Theorem 2.1]{cd15} shows
\begin{cor}\label{main-cor-2}
Let $\Omega$ be a bounded finitely connected domain in $\rr^2$.
Then the following statements are equivalent:

(i) $\Omega$ is a John domain;

(ii) The Korn inequality $(K_2)$ holds;

(iii) The Babu\v ska-Aziz inequality holds;

(iv) The Friedrichs inequality holds.
\end{cor}

Recall that, in \cite{hp83}, Horgan and Payne verified the equivalence of $(ii), (iii),(iv)$ on {\em simply connected regular domains} in the plane,
Costabel and Dauge  \cite{cd15} showed the equivalence of $(iii), (iv)$ on bounded domains without any a priori regularity assumption.
Our result above requires finite connectivity of underlying domains, and provides the extra geometric equivalence $(i)$.
However, the advantage of \cite{cd15,hp83} is that they have also the equivalence of constants.

The proof of Theorem \ref{main}, in particular, the proof of $(ii), (iii)\Longrightarrow (i)$,  will be divided into two steps. In the first step,
we will give a geometric characterization of the John domains via controlling the measure of ends by separating balls. We expect such a geometric characterization will have independent interest. In the second step, we first construct good test functions, and then use the Korn inequalities and this characterization of John domains to conclude the claim.

The paper is organized as follows. In Section 2, we recall some known result on the divergence equation, and provide the proofs for
the implications $(DE_p)\Longrightarrow(K_q), (\widehat{K}_q)$, $1/p+1/q=1$,  for completeness. In Section 3, we provide a geometric characterization of the John domains.
In Section 4 and Section 5, we will prove $(ii), (iii)\Longrightarrow (i)$
and complete the proof of Theorem \ref{main}. In section 6, we will discuss the necessity of  our assumption of the separation property and some related problems, and provide details for Example \ref{infinite-nonjohn} and Example \ref{simply-nonjohn}.

Throughout the paper, we denote by $C$ positive constants which
are independent of the main parameters, but which may vary from
line to line.  For matrices  $S\in \rr^{n\times n}$ we use the norm $|S|:=\max\{|s_{i,j}|: 1\le i,j\le n\}$.

\section{The divergence equation and Korn inequality}
\hskip\parindent In this section, let us recall the known results on the divergence equation and its applications to the Korn inequality. The following result was proved in \cite{adm06}; see also \cite{drs10} for a different
proof.
\begin{thm}[\cite{adm06}]\label{adm}
Let $q\in (1,\infty)$ and $\Omega \subset \rr^n$ be a bounded John domain. Then for each $f\in L^q_0(\Omega)$,
there exists a solution $\mathbf{v}\in W^{1,q}_0(\Omega,\rn)$
to the equation $\mathrm{div}\,\mathbf{v}=f$ with
$$\|\mathbf{v}\|_{W^{1,q}(\Omega,\rn)}\le C\|f\|_{L^q(\Omega)}.\leqno(DE_q)$$
\end{thm}

It is somehow standard to show that $(DE_q)$ implies the Korn inequalities $ (K_{p})$ and $(\widehat{K}_{p})$, where $1/p+1/q=1$.
We provide a proof for completeness.

\begin{lem}\label{john-korn}
Let $\Omega \subset \rr^n$ be a bounded domain, $p,q\in (1,\infty)$ be a H\"older conjugate pair.
If $(DE_q)$ holds on $\Omega$, then the Korn inequalities $ (K_{p})$ and $(\widehat{K}_{p})$ hold on $\Omega$.
\end{lem}
\begin{proof}
Recall that  $D\mathbf{u}=(\frac{\partial u_i}{\partial x_j})_{1\le i,j\le n}$, $1\le i,j\le n$,
and $\epsilon(\mathbf{u})=(\epsilon_{i,j}(\mathbf{u}))_{1\le i,j\le n}$ with
$$\epsilon_{i,j}=\frac 12\left(\frac{\partial u_i}{\partial x_j}+\frac{\partial u_j}{\partial x_i}\right)$$
and the identity
\begin{equation}\label{2.x}
\frac{\partial^2u_i}{\partial x_j\partial x_k}=\frac{\partial \epsilon_{i,k}(\mathbf{u})}{\partial x_j}
+\frac{\partial \epsilon_{i,j}(\mathbf{u})}{\partial x_k}-\frac{\partial \epsilon_{j,k}(\mathbf{u})}{\partial x_i}.
\end{equation}

For each $f\in L_0^q(\Omega)$, from assumption of $(DE_q)$, we see that there
exists a solution $\mathbf{v}\in W^{1,q}_0(\Omega,\rn)$
to the equation $\mathrm{div}\,\mathbf{v}=f$ with
$$\|\mathbf{v}\|_{W^{1,q}(\Omega,\rn)}\le C\|f\|_{L^q(\Omega)}.$$
For each $k\in \{1,2,\cdots,n\}$, let $\mathbf{v}^k$ be the $k$-th component  of
$\mathbf{v}$. Using the identity
\eqref{2.x} yields
\begin{eqnarray*}
 \lf|\int_\Omega f(x)\lf(\frac{\partial u_j}{\partial
x_i}(x)-\lf(\frac{\partial u_j}{\partial x_i}\r)_{\Omega}\r)\,dx\r|&&=
\lf|\int_\Omega \mathrm{div}\,\mathbf{v}(x) \lf(\frac{\partial u_j}{\partial
x_i}(x)-\lf(\frac{\partial u_j}{\partial x_i}\r)_{\Omega}\r)\,dx\r|\\
&&=\sum_{k=1}^n\lf|\int_{\Omega} \mathbf{v}^k(x) \lf(\frac{\partial
\epsilon_{j,k}(\mathbf{u})}{\partial x_i}
+\frac{\partial \epsilon_{j,i}(\mathbf{u})}{\partial x_k}-\frac{\partial \epsilon_{k,i}(\mathbf{u})}{\partial x_j} \r)(x)\,dx\r|\\
&&\le C\sum_{i,j,k=1}^n\int_\Omega \lf|\frac{\partial\mathbf{v}^k(x)}{\partial
x_i}\epsilon_{j,k}(\mathbf{u})(x)\r|\,dx\\
&&\le  C\|D\mathbf{v}\|_{L^q(\Omega)}\|\epsilon(\mathbf{u})\|_{L^p(\Omega)}\\
&&
\le C\|f\|_{L^q(\Omega)}\|\epsilon(\mathbf{u})\|_{L^p(\Omega)}.
 \end{eqnarray*}
A duality argument gives
\begin{eqnarray}\label{2.1}
 \lf\|\frac{\partial u_j}{\partial x_i}-\lf(\frac{\partial u_j}{\partial x_i}\r)_{\Omega}\r\|_{L^p(\Omega)}\le  C\|\epsilon(\mathbf{u})\|_{L^p(\Omega)}.
\end{eqnarray}

{\bf Proof of $(K_p)$}. Suppose now   $\int_\Omega \kappa_{i,j}(\mathbf{u})\,dx=0$ for $1\le i,j\le n$.
The \eqref{2.1} implies that
\begin{eqnarray*}
 \lf\|2\kappa_{i,j}(\mathbf{u})\r\|_{L^p(\Omega)}\le
  \lf\|\frac{\partial u_i}{\partial x_j}-\lf(\frac{\partial u_i}{\partial x_j}\r)_{\Omega}\r\|_{L^p(\Omega)}+
   \lf\|\frac{\partial u_j}{\partial x_i}-\lf(\frac{\partial u_j}{\partial x_i}\r)_{\Omega}\r\|_{L^p(\Omega)} \le  C\|\epsilon(\mathbf{u})\|_{L^p(\Omega)}.
\end{eqnarray*}
Then the Korn inequality $(K_p)$ follows since
\begin{eqnarray*}
 \lf\|D\mathbf{u}\r\|_{L^p(\Omega)}\le \|\epsilon(\mathbf{u})\|_{L^p(\Omega)}+ \|\kappa(\mathbf{u})\|_{L^p(\Omega)}\le  C\|\epsilon(\mathbf{u})\|_{L^p(\Omega)}.
\end{eqnarray*}


{\bf Proof of $(\widehat K_p)$}.
Now for an arbitrarily fixed cube $Q\subset\subset \Omega$, we choose a $\psi\in C_0^\fz(Q)$
such that $\supp \psi \subset Q$, $\int_Q\psi\,dx=1$ and $|\nabla \psi|\le C/\ell(Q)^{n+1}$.
Write
\begin{eqnarray}\label{2.2}
\frac{\partial u_j}{\partial x_i}=\frac{\partial u_j}{\partial x_i}-\lf(\frac{\partial u_j}{\partial x_i}\r)_{\Omega}
+\int_Q \lf[\lf(\frac{\partial u_j}{\partial x_i}\r)_{\Omega}-\frac{\partial u_j}{\partial x_i}\r]\psi\,dx+
\int_Q \frac{\partial u_j}{\partial x_i}\psi\,dx.
\end{eqnarray}
Then by the H\"older inequality, we obtain
\begin{eqnarray*}
\lf| \int_Q \lf[\lf(\frac{\partial u_j}{\partial x_i}\r)_{\Omega}-\frac{\partial u_j}{\partial x_i}\r]\psi\,dx \r|\le
C(p,Q,\Omega)  \|\epsilon(\mathbf{u})\|_{L^p(\Omega)},
\end{eqnarray*}
and
$$\lf|\int_Q \frac{\partial u_j}{\partial x_i}(x)\psi(x)\,dx\r|\le C(a,p,Q,\Omega)
\lf\|\frac{\partial u_j}{\partial x_i}\r\|_{L^p(Q)}.$$

Combining \eqref{2.1}, \eqref{2.2} and the above estimates, we obtain that
\begin{eqnarray*}
\lf\|\frac{\partial u_j}{\partial x_i}\r\|_{L^p(\Omega)}\le C(p,\Omega,Q)
\lf\{\|\epsilon(\mathbf{u})\|_{L^p(\Omega)}+\lf\|\frac{\partial u_j}{\partial x_i}\r\|_{L^p(Q)}\r\},
\end{eqnarray*}
which is
$$\lf\|D\mathbf{u}\r\|_{L^p(\Omega)}\le C(p,\Omega,Q)
\lf\{\|\epsilon(\mathbf{u})\|_{L^p(\Omega)}+\lf\|D\mathbf{u}\r\|_{L^p(Q)}\r\}.
\leqno (\widehat{K}_{p})$$
The proof is complete.
\end{proof}
Theorem \ref{adm} together with Lemma \ref{john-korn} gives the following corollary.
\begin{cor}\label{john-korn-main}
Let $\Omega \subset \rr^n$ be a bounded John domain. Then for each $p\in (1,\infty)$,
the Korn inequalities $ (K_{p})$ and $(\widehat{K}_{p})$ hold on $\Omega$.
\end{cor}

\section{A geometric characterization of John domains}
\hskip\parindent
In this section, assuming the separation property, we explore another characterization of the John condition.
We shall show that a domain is John if and only if the measures of ends are controlled with measures of separating balls. Similar arguments have been previously used to show necessity of John condition in other contexts in \cite{hk94} and \cite{bk95}. For the proof we need the following lemma  originally from \cite[Lemma 3.1]{ag85} and the well-known Whitney decomposition (see eg. \cite{s70}).

\begin{lem}[\cite{ag85}]\label{astala-gehring}
 Let $1\leq b<\infty$ and let $x_j$ be a sequence of nonnegative numbers such that for all $k\in \N$
 \begin{equation*}
  \sum ^\infty_{j=k} x_j \leq b x_k.
 \end{equation*}
 Then for every $\alpha\in \left(0,1\right]$  there  exists a constant $c\geq1$, depending only on $b,\alpha$, such that for all $k\in \N$
\begin{equation*}
  \sum ^\infty_{j=k} x_j^\alpha \leq c x_k^\alpha.
 \end{equation*}
\end{lem}

\begin{lem}\label{whitney}
For any open proper subset $\Omega\subset\rn$ there exists a collection $W =\{Q_j\}_{j\in\cn}$
of countably many closed dyadic cubes such that

(i) $\Omega=\cup_{j\in \cn}Q_{j}$, and if $k\neq j$, then the cubes have disjoint interiors, $(Q_{j})^\circ\cap (Q_{k})^\circ=\emptyset,$

(ii) $\sqrt n\ell(Q_{k})\le \dist(Q_{k}, \partial\Omega)\le 4\sqrt n\ell(Q_{k})$ and

(iii) $\frac{1}{4}\ell(Q_k)\le \ell(Q_j)\le 4\ell(Q_k)$ whenever $Q_k\cap Q_j\neq \emptyset$.
\end{lem}
Remember that Whitney decompositions are not unique. The following statement holds for any fixed $W.$

\begin{prop}\label{separation}
 Let $\Omega\subset\rn$ be a domain satisfying the separation property with constant $C_s\ge 1$ and a distinguished point
$x_0$.  For each point $x\in \Omega$, there is a curve $\gz$ connecting $x$ to $x_0$ that satisfies the separation property
with  constant $5C_s$, and for each Whitney cube $Q\in W$, the set $Q\cap \gz$ has at most one component.
\end{prop}
\begin{proof}
Let  $x\in \Omega$ be a point and let $\gz:[0, 1]\to \Omega$ with $\gz(0)= x$,
$\gz(1)= x_0$ be a curve given by the separation condition. We may assume the curve $\gz$ has no self-intersecting points,
otherwise, if there exist $0\le r<t\le 1$ such that $\gz(r)=\gz(t)$, then one can modify the curve $\gz$ as $\tilde \gz: [0,r]\cup [t,1]\mapsto \Omega$,
and $\tilde \gz$ is a curve that connects $x$ and $x_0$ satisfying the separation property with the constant $C_s$.

Now consider the collection $W_\gz$ of all Whitney
cubes that intersect $\gamma$. Notice that since $\gz$ is  a compact set in $\Omega$, such a  collection $W_\gz$
is  finite. If for each $Q\in W_\gz$, $Q\cap \gz$ has at most one
component, then we are done.
Otherwise, for each cube $Q$ for which $Q\cap \gz$ has more than one component,
let $t_u=\sup\{t:\gz(t)\in Q\}$ and $t_l=\inf\{t:\gz(t)\in Q\}$, and replace the curve $\gz([t_l,t_u])$
by the line segment $\ell_{\gz(t_l),\gz(t_u)}$. In such a way, we obtain a new curve $\tilde \gz$ connecting $x$ to $x_0$.
After reparametrisation, we denote the curve by $\tilde \gz:[0,1]\mapsto \Omega$.

{\bf Claim.} The curve $\tilde \gz$ connecting $x$ and $x_0$ satisfies the separation condition with constant $5C_s$.

In what follows, for convenience, denote $B(\gz(t),C_s\rho(\gz(t)))$ and  $B(\tilde \gz(t),5C_s\rho(\tilde\gz(t)))$ by $B_{\gz,t}$ and $B_{\tilde \gz, t}$ respectively.

Let $t\in [0,1]$. If $\tilde \gz([0,t])\subset B_{\tilde \gz,t}$ then the conclusion follows. Suppose that $\tilde \gz([0,t])\setminus B_{\tilde \gz,t}\neq\emptyset$.

{\bf Case 1.} $\tilde \gz(t)\in \gamma([0,1])$.

Let $t_\gz\in [0,1]$ be such that $\gz(t_\gz)=\tilde \gz(t)$.  Notice that
by the construction of $\tilde \gz$, $\tilde \gz([0,t])\setminus \gz([0, t_\gz])$ only contains line segments, whose endpoints belong to
$\gz([0, t_\gz])$. From this we see that  $\gz([0, t_\gz]) \setminus B_{\tilde \gz,t} \neq\emptyset$, otherwise we have
$\tilde \gz([0,t])\subset B_{\tilde \gz,t}$.

Since $B_{\tilde \gz,t}\supseteq  B_{\gz,t_\gz}$ and $\gz([0, t_\gz]) \setminus B_{\tilde \gz,t} \neq\emptyset$, we have
 $\gz([0, t_\gz]) \setminus B_{\gz,t_\gz} \neq\emptyset$.
By the separation property, each $y\in \gz([0, t_\gz]) \setminus B_{\gz,t_\gz} $  and $x_0$
belong to different component of $\Omega\setminus \partial B_{\gz,t_\gz} $.  Hence, we see that each
$$y\in (\gz([0,t_\gz])\cap \tilde \gz([0,t]))\setminus B_{\tilde \gz,t}\subset  (\gz([0,t_\gz])\cap \tilde \gz([0,t]))\setminus B_{\gz,t_\gz}\subset \gz([0, t_\gz]) \setminus B_{\gz,t_\gz}$$
belongs to a different component of $\Omega\setminus \partial B_{\tilde \gz,t}$ than $x_0$.

Now, if $(\tilde \gz([0,t])\setminus \gz([0, t_\gz])) \setminus  B_{\tilde \gz,t} \neq\emptyset$,
then there is  a line segment $\ell_{\gz(t_l),\gz(t_u)}\setminus  B_{\tilde \gz,t}\neq \emptyset$
and $\ell_{\gz(t_l),\gz(t_u)}\subset (\tilde \gz([0,t])\setminus \gz([0, t_\gz]))$. For each
$y\in \ell_{\gz(t_l),\gz(t_u)}\setminus  B_{\tilde \gz,t}$, there is one end-point of $\ell_{\gz(t_l),\gz(t_u)}$, which we  assume to be $\gz(t_l)$,
that also belongs to $ \ell_{\gz(t_l),\gz(t_u)}\setminus  B_{\tilde \gz,t}$ and connects to $y$ in the segment. Since
$\gz(t_l)\in \gz$ that belongs to a different component of $\Omega\setminus \partial B_{\tilde \gz,t}$ than $x_0$,
 we find that the  set $\ell_{\gz(t_l),\gz(t_u)}\setminus  B_{\tilde \gz,t}$ belong to a different component of $\Omega\setminus \partial B_{\tilde \gz,t}$ than $x_0$, as desired.

{\bf Case 2.} $\tilde \gz(t)\in \tilde \gz([0,1])\setminus \gz([0,1])$.

 By the construction of the curve $\tilde \gz$, we see that there
is a line segment $\ell_{\gz(t_l),\gz(t_u)}\subset \tilde \gz([0,1])$, that contains $\tilde \gz(t)$ and belongs to a Whitney cube $Q$.
Notice that the length of $\ell_{\gz(t_l),\gz(t_u)}$
is no bigger than $\sqrt n\ell({Q})$, which is no bigger than $\rho(\tilde\gz(t))$ or $\ro(\gz(t_l))$. Therefore, $\ell_{\gz(t_l),\gz(t_u)}\subset B_{\tilde \gz,t}$,
$$\rho(\gz(t_l))\le d(Q,\partial \Omega)+\sqrt n\ell(Q)\le 2\rho(\tilde \gz(t))$$ and
$$B_{\gz,t_l}=B(\gz(t_l), C_s\rho(\gz(t_l)))\subset B(\tilde \gz(t), 5C_s\rho(\tilde\gz(t)))=B_{\tilde \gz,t}.$$

If $y\in \tilde \gz([0,t])\setminus B_{\tilde \gz,t}$ and $y\in \gz([0,1])$, then
$$y\in \gz([0,t_l]) \setminus B_{\tilde \gz,t}\subset  \gz([0,t_l]) \setminus B_{\gz,t_l}.$$
Since $\gz$ satisfies the separation property, this implies that $y$ and  $x_0$ belong to different components of $\Omega\setminus \partial B_{\gz,t_l}$, and further
$y$ and $x_0$ belong to different components of
$\Omega\setminus \partial B_{\tilde\gz,t}$.

If $y\in \tilde \gz([0,t])\setminus B_{\tilde \gz,t}$ and $y\notin \gz([0,1])$, then by the construction of the curve $\tilde \gz$ again,
we see that there is a line segment $\ell_{\gz(\tilde t_l),\gz(\tilde t_u)}\subset \tilde \gz([0,t])$ such that
$y\in \ell_{\gz(\tilde t_l),\gz(\tilde t_u)}$.
Arguing as in {\bf Case 1}, we see that there is one end-point of $\ell_{\gz(\tilde t_l),\gz(\tilde t_u)}$, assuming that it is $\gz(\tilde t_l)$,
that belongs to  $\tilde \gz([0,t])\setminus B_{\tilde \gz,t}$ and connects to $y$ in the segment.

Since
 $\gz(\tilde t_l)\in \tilde \gz([0,t])\setminus B_{\tilde \gz,t}$ and $\gz(\tilde t_l)\in\gz([0,1])$, we can conclude  that
$y$ and  $\gz(\tilde t_l)$ belong to the same component of $\Omega\setminus \partial B_{\tilde\gz,t}$, while $x_0$
belongs to another component, as desired.
\end{proof}

\begin{thm}\label{john-geo}
 Let $\Omega\subset\rn$ be a domain satisfying the separation property. Then
 $\Omega$ is a John domain if and only if there exists a positive constant $C_E$ such that for every separating ball $B$,
  it holds that
  \begin{equation}
   \label{measurecondition}
\abs{E_B}\leq C_E \abs{B}
\end{equation}
for $B-$end $E_B.$
\end{thm}

\begin{proof}
It is quite standard to prove \eqref{measurecondition} assuming John condition (see \cite[Theorem 2.1]{bk95}).
For the converse implication we modify the ideas from \cite[p. 18]{kos15} and \cite{bk95}.

Let  $x\in \Omega$ be a point and let $\gz:[0, 1]\to \Omega$ with $\gz(0)= x$,
$\gz(1)= x_0$ be a curve given by the separation condition. From the previous proposition,
we may assume that for each Whitney cube $Q$ the intersection $Q\cap \gz$ has at most one component.

To prove that $\Omega$ is a John domain it suffices to show that
$$\rho(\gz(t))=d(\gamma(t),\partial \Omega)\geq C \diam(\gamma([0,t])),$$
 see \cite[pp.385-386]{ms79} or \cite[pp.7-8]{nv91}.

If $x\in B(x_0,\rho(x_0))$, then one can take the line segment connecting $x$ and $x_0$, and the conclusion is obvious.
Suppose now $x\notin B(x_0,\rho(x_0))$.  Consider the collection $W_\gz$ of all Whitney
cubes that intersect $\gamma$. Since for each Whitney cube $Q$ in $W_\gz$,  there is only one
component in $Q\cap\gz$ by our reduction above,  we can order them so that they
form a chain $\{Q_j\}_{j=0}^{W_x}$ with $x_0\in   Q_0$ and $W_x\in \cn$ depends on $x$.
Notice that each cube is only numbered once.

For each $t\in [0,1]$, let $k_t$ ($0\le k_t\le W_x$) be the smallest number such that $\gz(t)\in Q_{k_t}$.
For the point  $\gz(t)$, $t\in [0,1]$,
the separation condition implies either $\gz([0, t])\subset B_{\gz,t}$,  $B_{\gz,t}:=B(\gz(t),C_s\rho(\gz(t)))$, or  every $y\in \gamma([0,t])\setminus B_{\gamma,t}$ and $x_0$ lie in different components of $\Omega\setminus \partial B_{\gz,t}$ than $x_0$.

{\bf Case 1.} Suppose that the former case happens.  Then $\gz([0, t])\subset B_{\gz,t}$, and we have  for each $r<t$ that
$$\rho(\gz(r))\le  C_s\rho(\gz(t))+\rho(\gz(t)),$$
and for each $Q_j\in W_\gz$ that contains $\gz(r)$,
\begin{equation}\label{3.1}
Q_j\subset B(\gz(r), \sqrt n\ell(Q_j))\subset B(\gz(r), \rho(\gz(r)))\subset B(\gz(t),(2C_s+1)\rho(\gz_t))\subset B(\gz(t),3C_s\rho(\gz_t)).
\end{equation}
This together with the definition of the chain $\{Q_j\}_{j=0}^{W_x}$ implies that
$\cup_{j=k_t}^{W_x}Q_j\subset B(\gz(t),3C_s\rho(\gz_t)),$ and hence,
\begin{equation}
\label{cond1}
 \sum_{j=k_t}^{W_x} \ell(Q_j)^n \leq C(n,C_s) \ro(\gz(t))^n\leq C(n,C_s) \ell(Q_{k_t})^n.
\end{equation}

{\bf Case 2.} Suppose that every $y\in \gamma([0,t])\setminus B_{\gz,t}$ lies in a different component of
$\Omega\setminus \partial B_{\gz,t}$ than $x_0$.  We claim that in this case
$$\bigcup_{j=k_t}^{W_x}Q_j\subset E_{B_{\gz,t}}\cup B(\gz(t), (3+2\sqrt n C_E)C_s\rho(\gz_t)).$$
Notice that in this case we have $\gamma([0,t])\subset E_{B_{\gz,t}}\cup B_{\gz,t}.$
For the cubes $\{Q_j\}_{j=k_t}^{W_x}$, if $Q_j\cap \gamma([0,t]) \cap B_{\gz,t}\neq\emptyset$,
then  the same argument as in proving \eqref{3.1} gives that
$$Q_j\subset  B(\gz(t), 3C_s\rho(\gz_t))\subset E_{B_{\gz,t}}\cup B(\gz(t), (3+2\sqrt n C_E)C_s\rho(\gz_t)).$$

For each cube $Q_j$ that $Q_j\cap \gamma([0,t]) \cap B_{\gz,t}=\emptyset$,   there is $r\in [0,t]$ such that $\gz(r)\in Q_j$ and $\gz(r)\in E_{B_{\gz,t}}$.
If  $Q_j\subset E_{B_{\gz,t}}$, then we also have $Q_j\subset E_{B_{\gz,t}}\cup B(\gz(t), (3+2\sqrt n C_E)C_s\rho(\gz_t))$.

It remains to consider the case $Q_j\setminus E_{B_{\gz,t}}\neq \emptyset$.
Notice that now $Q_j\cap B_{\gz,t}\neq\emptyset$. From Lemma \ref{whitney} and $\gz(r)\in Q_j$, we deduce that
$$\sqrt n\ell({Q_j})\le \rho(\gz(r))\le d(Q_j,\partial\Omega)+\sqrt n\ell{(Q_j)}\le 5\sqrt n \ell(Q_j).$$
This implies that $B(\gz(r),\ell({Q_j}))$ is compactly contained in $\Omega$, further, since the center  $\gz(r)$ is in
$E_{B_{\gz,t}}$, we can conclude that at least half of $B(\gz(r),\ell({Q_j}))$ belongs to $E_{B_{\gz,t}}$,
and from the assumption that
$$\frac 12|B(\gz(r),\ell({Q_j}))|\le |B(\gz(r),\ell({Q_j}))\cap E_{B_{\gz,t}}|\le |E_{B_{\gz,t}}| \le C_E|B_{\gz,t}|,$$
and therefore $\ell(Q_j)\le (2C_E)^{1/n}C_s\rho(\gz(t))$.
This and $Q_j\cap B_{\gz,t}\neq \emptyset$ give $Q_j\subset  B(\gz(t),2\sqrt n C_E C_s\rho(\gz_t)) $, and hence completes the proof
of
$$\bigcup_{j=k_t}^{W_x}Q_j\subset E_{B_{\gz,t}}\cup B(\gz(t), (3+2\sqrt n C_E)C_s\rho(\gz_t)).$$

Finally, by the non-overlap property of Whitney cubes, Lemma \ref{whitney},
we conclude that
\begin{eqnarray}\label{cond2}
 \sum_{j\geq k_t}\ell(Q_j)^n&&=\sum_{j=k_t}^{W_x} |Q_j|\leq \abs{E_{B_{\gz,t}}} + \abs{B(\gz(t), (3+2\sqrt n C_E)C_s\rho(\gz_t))}\nonumber\\
 &&\leq C(C_s,n,C_E)\rho(\gamma(t))^n\leq C(C_s,n,C_E)\ell(Q_{k_t})^n.
\end{eqnarray}

The estimates \eqref{cond1} and \eqref{cond2} with Lemma \ref{astala-gehring} imply  that for all $t\in [0,1]$
\begin{equation}
 \sum_{j\geq k_t}\ell(Q_j)\leq C \ell(Q_{k_t}).
\end{equation}
Since the cubes $\{Q_i\}_{i\ge k_t}$ cover the curve $\gamma([0,t])$, we find that
$\diam \gamma([0,t])\leq C \rho(\gamma(t))$ for all $t\in [0,1]$, as desired.
\end{proof}

\section{Korn inequality $(K_p)$ implies John condition}
\hskip\parindent  In this and the next section, we provide the proof of the main result of the paper.
\begin{thm}\label{korn-john-2}
Let $\Omega$ be a bounded domain of $\rr^n$, $n\ge 2$. Let
$1<p<\infty$.
Suppose that
for all $\mathbf{v}\in W^{1,p}(\Omega,\rn)$ satisfying
$\int_\Omega \kappa_{i,j}(\mathbf{v})\,dx=0,$ $1\le i,j\le n$,
it holds that
$$\lf\|D\mathbf{v}\r\|_{L^p(\Omega)}\le C_K\|\epsilon(\mathbf{v})\|_{L^p(\Omega)}.
\leqno ({{K}_{p}})$$
Then if  $\Omega$  satisfies the separation property, $\Omega$ is a John domain.
\end{thm}

\begin{proof}Suppose that $\Omega$  satisfies the separation property w.r.t. $x_0\in \Omega$.
Let $Q$ be the Whitney cube such that $x_0\in Q\subset\subset \Omega$.

By Theorem \ref{john-geo}, to show that $\Omega$ is a John domain,
it suffices to show that there exists an absolute constant $C_E>0$
such that for each separating ball $B$, $B=B(z,r)$ with $z\in \Omega$, its end $E_B$,
if exists, has the property  $|E_B|\le C_E|B|$.

If $B\cap Q\neq \emptyset$, then by $B\cap \partial \Omega\neq \emptyset$ we see that  $r\ge \sqrt n\ell(Q)/2$.
In this case, it holds that $|E_B|\le |\Omega|\le C(\Omega,Q)|B|$.

Suppose now that $B\cap Q=\emptyset$. Let $E_B$ be the $B$-end.  If
$|B|\ge {|Q|}/({C_K^p3^{n+2p}}),$
then it holds that
$|E_B|\le |\Omega|\le C(\Omega,Q,C_K)|B|.$
Therefore we may assume that
$$|B|< \frac{|Q|}{C_K^p3^{n+2p}}.$$

If $E_B\subset B(z,4r)$ or $|E_B|\le 4^n|B|$ then the conclusion is obvious. Otherwise,
set
\begin{equation}\label{test-phi}
\phi(x):=\ \
\begin{cases}
0, \hspace{2cm}  \forall x\in \Omega\setminus E_B;\\
1, \hspace{2cm}  \forall x\in E_B\setminus B(z,2r);\\
\frac{d(x,B(z,r))}{r}, \hspace{0.8cm} \forall x\in E_B\cap (B(z,2r)\setminus B(z,r)).
\end{cases}
\end{equation}
Then $\phi$ is a Lipschitz function, with Lipschitz constant being $1/r$, that vanishes on $B\cap \Omega$.

For each $x=(x_1,\cdots,x_n)\in \Omega$, let $\mathbf{v}=(v_1,v_2,0,\cdots,0)$ with
\begin{equation}\label{test-v}
\lf\{
\begin{array}{ccc}
v_1(x_1,\cdots,x_n)&=& (x_2-z_2) \phi(x_1,\cdots,x_n),\\
v_2(x_1,\cdots,x_n)&=&(z_1-x_1)\phi(x_1,\cdots,x_n),
\end{array}
\r.
\end{equation}
where $z=(z_1,\cdots,z_n)$ is the center of $B$.
Then for each $x=(x_1,\cdots,x_n)\in E_B\setminus B(z,2r)$,
\begin{eqnarray}\label{value-v}
D\mathbf{v}(x)=\lf(
\begin{array}{ccccc}
0 &  1 & 0& \cdots & 0\\
-1 & 0 & 0&\cdots & 0\\
0 & 0 & 0&\cdots & 0\\
\cdots\\
0 & 0 & 0&\cdots & 0
\end{array}
\r),
\end{eqnarray}
$D\mathbf{v}(x)= 0$ for all $x\in \Omega\setminus E_B$ and
$$|D\mathbf{v}(x)|\le 2r|\nabla \phi(x)|+\phi(x)\le 3$$
for all $x\in E_B\cap (B(z,2r)\setminus B(z,r))$.

Since now $|E_B|>4^n|B|$, $|E_B\setminus B(z,2r)|\ge |E_B|-|2B|>3|B(z,2r)|$.
Then from the construction of $\mathbf{v}$ we conclude that
$$\int_\Omega \frac{\partial v_1}{\partial x_2}-\frac{\partial v_2}{\partial x_1}\,dx=2\int_{E_B\setminus B(z,2r)}\,dx+\int_{E_B\cap B(z,2r)} \frac{\partial v_1}{\partial x_2}-\frac{\partial v_2}{\partial x_1}\,dx>0,$$
since $\int_{E_B\cap B(z,2r)} \frac{\partial v_1}{\partial x_2}-\frac{\partial v_2}{\partial x_1}\,dx>-6|B(z,2r)|$.
Choose a vector field $\mathbf{w}$ on $\Omega$ as
$$\mathbf{w}(x)=\mathbf{w}(x_1,\cdots,x_n)=(-\tilde Cx_2,\tilde Cx_1,0,\cdots,0),$$
where $\tilde C$ satisfies
$$2\tilde C|\Omega|=\int_\Omega \frac{\partial v_1}{\partial x_2}-\frac{\partial v_2}{\partial x_1}\,dx.$$
Now set $\mathbf{u}=\mathbf{v}+\mathbf{w}$. One has that  $\mathbf{u}=(u_1,u_2,0,\cdots)$,
where $u_i=v_i+w_i$, $i=1,2$, is Lipschtiz continuous on $\Omega$
and satisfies
$$\int_\Omega \frac{\partial u_1}{\partial x_2}-\frac{\partial u_2}{\partial x_1}\,dx=\int_\Omega \frac{\partial v_1}{\partial x_2}-\frac{\partial v_2}{\partial x_1}\,dx-2\tilde C|\Omega|=0.$$
Applying the Korn inequality $({{K}_{p}})$ to $\mathbf{u}$, and noticing that $\epsilon(\mathbf{w})\equiv 0$, we obtain
\begin{equation}\label{est-gra}
\|D\mathbf{u}\|_{L^p(\Omega)}\le C_K\|\epsilon(\mathbf{u})\|_{L^p(\Omega)}=C_K\|\epsilon(\mathbf{v})\|_{L^p(\Omega)}\le 3C_K|B(z,2r)|^{1/p}.
\end{equation}
By the construction of $\mathbf{v}$ we have $\mathbf{v}=0$ on $\Omega\setminus E_B$, and therefore,
\begin{equation}
\tilde C|Q|^{1/p}\le \tilde C|\Omega\setminus E_B|^{1/p}\le \|D\mathbf{u}\|_{L^p(\Omega)}\le 3C_K|B(z,2r)|^{1/p}< 3C_K\left(\frac{2^n|Q|}{C_K^p3^{n+2p}}\right)^{1/p}.
\end{equation}
From this, we see that
$\tilde C<1/3$.
By this, \eqref{est-gra} and \eqref{value-v},
we can conclude that
\begin{eqnarray*}
\frac{2^p}{3^p}|E_B\setminus B(z,2r)|&&\le\lf\|D\mathbf{u}\r\|_{L^p(\Omega)}^p \le 3^pC_K^p|B(z,2r)|.
\end{eqnarray*}
Hence, we find that $|E_B\setminus B(z,2r)|\le 3^{2p+n}C_K^p|B|$ and hence, $|E_B|\le C(n,C_K,p)|B|$.
Now applying Theorem \ref{john-geo}, we finally
obtain that $\Omega$ is a John domain.
\end{proof}

\section{Korn inequality $(\widehat{K}_p)$ implies John condition}
\hskip\parindent In this section,  we show that the Korn inequality $(\widehat{K}_p)$ also implies John condition, and therefore
complete the proof of Theorem \ref{main}. This requires little more work than the first implications.
For the proof we need the following lemma which is well-known for the experts but it seems that there are no proofs published so far.
\begin{lem}
 Let $\Omega\subset \rr^n$ be a domain satisfying the separation property with constant $C_s$ with respect to the distinguished point $x_0.$
Let $x\in\Omega.$ Then $\Omega$ satisfies the separation property also with respect to the distinguished point $x.$
\end{lem}
\begin{proof} We claim that if the point $x_0$ qualifies for a distinguished point  in the separation property,
then every point in $B(x_0,\rho(x_0)/1000)$ qualifies for a distinguished point  with constant $2C_s.$

By connectedness of $\Omega$, it suffices to prove this claim.
Indeed, letting $E$ be the set of all points qualifying as distinguished points, the  claim yields that $E$ is open.
 Suppose that $F:=\Omega\setminus E\neq \emptyset$. Then there exists $y_0\in  \partial E\cap F$. Take $\tilde y_0\in E$ with $|y_0-\tilde y_0|<\frac{\rho(y_0)}{10000}$.
Then
$$|y_0-\tilde y_0|<\frac{\rho(y_0)}{10000}<\frac{1}{10000}(\rho(\tilde y_0)+|y_0-\tilde y_0|),$$
and $|y_0-\tilde y_0|<\frac{\rho(\tilde y_0)}{9999}$. This implies that $y_0\in E$, since $y_0\in B(\tilde y_0,\frac{\rho(\tilde y_0)}{1000})$. This is a contradiction with our choice of $y_0$, therefore, $F=\emptyset$, and $E=\Omega$.

Let us prove the claim.
Take a point $x\in B(x_0,\rho(x_0)/1000).$ For any point $y\in \Omega$ there exists a curve $\tilde{\gamma}$ joining $x_0$ to $y$ satisfying the condition in the definition of the separation condition. Denote by $\gamma$ the composition of
$\tilde{\gamma}$ and line segment $\ell_{x,x_0}$ joining $x$ and $x_0.$ After a reparametrization we have $\gamma(0)=y$ and $\gamma(1)=x.$

Take a point $w=\gamma(t)$ on the curve $\gamma.$ If this point lies on the line segment $\ell_{x,x_0}$ then we
have $\rho(w)\geq \frac12 \rho(x_0)$ and, therefore, $x,x_0\in B(w,2 C_s \rho(w)).$ This implies the required condition.

If the point $w$ lies on the curve $\tilde{\gamma}$ then we have $w=\tilde{\gamma}(t')$ for some  $t'.$ Since
$\tilde{\gamma}$ is the curve given by the separation condition we have either $\tilde{\gamma}([0,t'])\subset B(w,C_s
\rho(w))$ or every point in $\tilde{\gamma}([0,t'])\setminus B(w,C_s \rho(w))$ belongs to a different component of
$\Omega\setminus \partial B(w,C_s \rho(w))$ than $x_0.$ In the former case we have $\tilde{\gamma}([0,t'])\subset B(w,2C_s
\rho(w)).$

In the latter case, there is nothing to prove if  $x \in B(w,2C_s \rho(w))$. For the remaining case, i.e.,   $x \notin B(w, 2C_s \rho(w))$,
let us argue by contradiction. Assume there exists  $z\in \tilde{\gamma}([0,t'])\setminus B(w,2C_s \rho(w))$ belongs to the same
component of  $\Omega\setminus \partial B(w,2C_s \rho(w))$ as $x$. This implies that $x$ and $z$ belong to a common component of
$\Omega\setminus \partial B(w,C_s \rho(w))$,
and furthermore,  $x$ and $z$ belong to a different component than $x_0$.
Since the line segment $\ell_{x,x_0}$ is contained in $\Omega$, it follows $\ell_{x,x_0}\cap\partial B(w,C_s \rho(w))\neq\emptyset$. Since $x \notin B(w, 2C_s \rho(w))$, we find that
\begin{equation}
\label{sepestimate}
C_s \rho(w)\le \abs{x-x_0}<\rho(x_0)/1000.
\end{equation}

On the other hand, $B(w, C_s \rho(w))$ intersects $\partial\Omega$ and, thus, it holds $\rho(x_0)\le |x_0-w|+C_s\rho(w).$ This with \eqref{sepestimate} gives
$ |x_0-w|>999\rho(x_0)/1000$.

Notice that now, $B(w,2C_s\rho(w))\cap B(x_0,\rho(x_0)/2)=\emptyset$, since for each $\tilde z\in B(x_0,\rho(x_0)/2)$, it holds
$$|\tilde z-w|\ge |w-x_0|-|x_0-\tilde z|>999\rho(x_0)/1000-\rho(x_0)/2>2\rho(x_0)/5>2C_s \rho(w).$$
This is a contradiction with  $\ell_{x,x_0}\cap \partial B(w,C_s \rho(w))\neq \emptyset$ and $\ell_{x,x_0}\subset B(x_0,\frac {\rho(x_0)}{1000})$.  Therefore, there
 is no such $z$, that is, each  $z\in \tilde{\gamma}([0,t'])\setminus B(w,2C_s \rho(w))$ belongs to a different
component of  $\Omega\setminus \partial B(w,2C_s \rho(w))$ than $x$. The proof is complete.
\end{proof}

With the help of above lemma, we are able to complete our proof.
\begin{thm}\label{t2.2}
Let $\Omega$ be a bounded domain of $\rr^n$, $n\ge 2$. Let
$1<p<\infty$ and $Q\subset\subset \Omega$ be a closed cube.
Suppose that
for all $\mathbf{v}\in W^{1,p}(\Omega,\rn)$ it holds that
$$\lf\|D\mathbf{v}\r\|_{L^p(\Omega)}\le C\lf\{\|\epsilon(\mathbf{v})\|_{L^p(\Omega)}+\lf\|D\mathbf{v}\r\|_{L^p(Q)}\r\}.
\leqno ({\widehat{K}_{p}})$$
Then if  $\Omega$  satisfies the separation property, $\Omega$ is a John domain.
\end{thm}
\begin{proof}
Suppose that $\Omega$  satisfies the separation property w.r.t. $x_0\in \Omega$.  Then by the above lemma,
$\Omega$  satisfies the separation property w.r.t. $x_Q\in Q$, where $x_Q$ is the center of $Q$.

By Theorem \ref{john-geo}, to show $\Omega$ is a John domain,
it suffices to show that there exists an absolute constant $C_E>0$
such that for each separating ball $B$, $B=B(z,r)$ with $z\in \Omega$, its end $E_B$, if exists, then $|E_B|\le C_E|B|$.

If $B\cap Q\neq \emptyset$, then by $B\cap \partial \Omega\neq \emptyset$ we see that  $r\ge \dist(Q,\partial \Omega)/2$.
In this case, it holds that $|E_B|\le |\Omega|\le C(\Omega,Q)|B|$.

Suppose now $B\cap Q=\emptyset$.
Let $E_B$ be the $B$-end. If $E_B\subset B(z,4r)$ or $|E_B|\le 4^n|B|$, then the conclusion is obvious.
Otherwise, let $\phi$ and $\mathbf{v}$ be given as in \eqref{test-phi} and \eqref{test-v}, respectively.

Applying the Korn inequality $({\widehat{K}_{p}})$ to $\mathbf{v}$, we find that
 $$\lf\|D\mathbf{v}\r\|_{L^p(\Omega)}\le C_K\lf\{\|\epsilon(\mathbf{v})\|_{L^p(\Omega)}+\lf\|D\mathbf{v}\r\|_{L^p(Q)}\r\},
\leqno ({\widehat{K}_{p}})$$
 and can conclude from the construction of $\mathbf{v}$ that
\begin{eqnarray*}\label{end-est}
|E_B\setminus B(z,2r)|&&\le\lf\|D\mathbf{v}\r\|_{L^p(\Omega)}^p \le
C_K^p\lf\{\|\epsilon(\mathbf{v})\|_{L^p(\Omega)}+\lf\|D\mathbf{v}\r\|_{L^p(Q)}\r\}^p\\
&&\le C_K^p\|\epsilon(\mathbf{v})\|_{L^p(B(z,2r))}^p\le 3^pC_K^p|B(z,2r)|.
\end{eqnarray*}
Hence, we find that $|E_B\setminus B(z,2r)|\le C|B|.$  From this and applying Theorem \ref{john-geo}, we finally
obtain that $\Omega$ is a John domain.
\end{proof}

\begin{proof}[Proof of Theorem \ref{main}]
The implication $(i)\Longrightarrow(iv)$  is contained in \cite{adm06} (see Theorem \ref{adm}).
The implication $(iv)\Longrightarrow(i)$ can be seen as follows:
if $(iv)$ holds for some $p\in (1,\infty)$, then Lemma \ref{john-korn} implies $(K_q)$ on $\Omega$, $1/p+1/q=1$,
and therefore $(i)$ holds by Theorem \ref{korn-john-2}.

The implications $(i)\Longrightarrow(ii), (iii)$ are rather standard and are proved in Corollary \ref{john-korn-main}.
The implications  $(ii), (iii)\Longrightarrow (i)$ are proved in Theorem \ref{korn-john-2}
and Theorem \ref{t2.2}, respectively.

It is obvious that $(ii'), (iii'), (iv')$ implies $(ii), (iii), (iv) $, respectively. Conversely,
if one of conditions $(ii), (iii)$ and $(iv)$ holds,
then (i) holds, i.e., $\Omega$ must be John, Theorem \ref{adm} and Corollary \ref{john-korn-main} imply that
$(ii'), (iii'), (iv')$ all hold.
\end{proof}

\section{Further discussions}
\label{discussion}

\subsection{The separation property}
\hskip\parindent In this subsection, we discuss  the separation
property and list several known examples of domains satisfying and not
satisfying this property.

It  follows quite easily from the definitions that every John domain has the
separation property.
On the other hand,  domains with exterior cusps, e.g.,
$$\Omega_\alpha=\left\{x=(x_1,\cdots,x_n):\, 0<x_n<1,
0<x_1^2+\cdots+x_{n-1}^2<x_n^{2\alpha}\right\} ,$$
with $n\ge 2$ and $\alpha>1$  or domains of  rooms-and-corridors type (cf.
\cite[Example 4.1]{jk15}), satisfy the separation property, but are not John
domains. Domains with inward cusps such as $B(0,1)\setminus
\{(x,0,\ldots,0)\colon x\geq0\}\subset \rr^n$ is John and thus also satisfies
separation condition.

A domain with flat exterior cusp
$$\Omega'_\alpha=\left\{x=(x_1,\cdots,x_{n+1}):\, 0<x_n, x_{n+1}<1,
0<x_1^2+\cdots+x_{n-1}^2<x_n^{2\alpha}\right\} \subset \rr^{n+1}$$
does not satisfy John condition nor separation property.

More generally, one has
\begin{lem}[Buckley-Koskela \cite{bk95}]\label{lem-separation}
Suppose that $\Omega\subset \mathbb{R}^n$ is quasiconformally equivalent to a uniform
domain $G\subset\mathbb{R}^n$, $n\ge 2$. Then $\Omega$ has the separation property.
\end{lem}

Above, by $\Omega$ quasiconformally equivalent to $G$, we mean there is a
homeomorphism $f$ of $G\subset \mathbb{R}^n$ onto $\Omega\subset\mathbb{R}^n$
with
$f\in W^{1,n}_\loc(G,\rn)$ for which there exists a constant $K\geq1$  such that
$|Df(x)|^n \le  KJ_f (x)$ for almost every $x\in G$,
here $J_f$ is
the
Jacobian determinant of $Df.$
If $K =1$ then the resulting mappings are conformal mappings. Other examples of
quasiconformal mappings are bi-Lipschitz mappings in all dimensions.

A domain $G$ is {\em uniform} if there is a constant $C>0$ such that for any pair $x, y\in G$
there exists  a curve $\gamma: [0,\ell]\mapsto  G$ parametrised by arclength
such that $\gamma(0) = x$, $\gamma(\ell) = y$, $\ell\le C|x-y|$, and $d(\gamma(t),\rn\setminus G) \ge
\frac 1C \min\{t, \ell-t\}$; see \cite{bk95,jo81} for more about uniform
domains. Typical examples of a uniform domains are the balls in $\rr^n.$
Therefore, all quasiconformal images of balls have separation
property. Notice that a uniform domain is always a John domain, but  the
converse is not true.

The following result is somehow well known, we include a proof for completeness; see \cite{bk95}.
\begin{cor}\label{fint-sp}
Any finitely connected plane domain satisfies the separation property.
\end{cor}
\begin{proof}
By Koebe's uniformization theorem, any finitely connected plane domain is
conformally equivalent to a circle domain,
where a circle domain means each component of its boundary is either a point or a circle; see \cite{hk91} for instance.
The above lemma together with Jones \cite[Theorem 4]{jo81} then gives the desired conclusion.
\end{proof}

This corollary together with Theorem \ref{main} gives Corollary \ref{main-cor}
for any finitely connected domains in the plane. On the other hand,
Example \ref{infinite-nonjohn} shows analogues of Corollary \ref{main-cor} fails
for infinitely connected domains in the plane.
{ Notice that an infinitely connected domain might  be or be not a John domain. For instance, on the plane, 
 a domain obtained by removing a countable set $F$ from the unit ball $B(0,1)$, 
 where $F=\cup_{k\in \cn}\{(1-2^{-k},0)\}$, is a John domain, while the domain $\Omega$ from the following example
is not a John domain. }

We remark that the domains
 used in the following two examples are from Buckley-Koskela \cite{bk95}.

\begin{proof}[Example \ref{infinite-nonjohn}] Take a ball $B=B(0,1)$. For each $k\in \cn$, let
$E_k =\{x_{k,j}\}_{j=1}^{k!}$,  where $x_{k,j}$ are equally spaced on the circle  $S(0,1-2^{-k})\subset \rr^2$.
Let $\Omega:=B\setminus E$, where $E=\cup_{k\in\cn}E_k$.
We claim that $\Omega$ is an infinitely connected plane domain, which supports the Korn inequality, but is not a John domain.

Notice that $E$ consists of countable points, and hence is removable for Sobolev spaces $W^{1,p}$; see Koskela \cite{kos99}. Since
$B=B(0,1)$ supports the Korn inequality, $E$ is removable, we see that $\Omega$ supports the Korn inequality as well.
However, $\Omega$ is not John. To see this, let us argue by contradiction. Suppose that $\Omega$ is a John domain
with a distinguished point $x_0$. Then $\Omega$ is a John domain with the distinguished point as the origin.
Take an arbitrary point $y_k\in \Omega$ with $|y_k|=1-2^{-k}$. Then for any curve $\gamma_k$ linked $y_k$ to the origin,
there is a point $z_{k}\in \gamma_k\cap S(0,1-2^{-k+1})$, then
$\mbox{lengh}(\gamma([y_k,z_k]))\ge 2^{-k}$, while $\rho(z_k)=d(z_k,\rr^2\setminus \Omega)\le d(z_k, E_k)\le \frac{C}{k!}$. This implies,
there cannot exist a constant $C_J>0$ such that
$$\frac{C}{k!}\ge \rho(z_k)\ge C_J\mbox{lengh}(\gamma([y_k,z_k]))\ge C_J2^{-k},$$
since $2^k/k!\to 0$ as $k\to \infty$. Therefore, $\Omega$ is not John, but supports the Korn inequality.
\end{proof}

Regarding higher dimensional cases, one cannot hope for exact analogues of Corollary \ref{main-cor} in dimensions $n\ge 3$
as we explained in the introduction; see \cite{jo81} and \cite{bk95}. We next complete the construction of Example \ref{simply-nonjohn}.

\begin{proof}[Example \ref{simply-nonjohn}] Let $\Delta=B(0,1)\subset \rr^2$, and $E=\cup_{k\in\cn}E_k$ be as in
the previous proof.
Let $F=E\times [0,1)$ and $D=\Delta\times (-1,1)$. Consider the domain $\Omega:=D\setminus F\subset \mathbb{R}^3$.
The domain $\Omega$ is then a simply connected domain in $\mathbb{R}^3$. Arguing similar to the previous proof,
one can see that $\Omega$ supports the Korn inequality, but $\Omega$ is not a John domain.
Similarly, one can construct examples in all dimensions $n\ge 3$.
\end{proof}

 Example \ref{simply-nonjohn} illustrates that the connectivity is not suitable for characterization of Korn inequality in high dimensional cases.
The separation property turns out to be a natural requirement in the sense that, it allows us to
get an almost complete classification of domains which supports Korn inequality in the plane,
and works well also in the higher dimensions.

Using the theory of removable sets for Sobolev functions, one can also show that for the domains in Examples \ref{infinite-nonjohn} and \ref{simply-nonjohn}, $(DE_p)$ for the divergence equation holds for $p\in (1,\infty)$.  It would be interesting to know if there are domains, without the separation property, such that the Korn inequality holds but $(DE_p)$ fails for some $p\in (1,\infty)$.

Notice that, in higher dimensions,  there is a rich class of domains  satisfying the separation property.
Besides John domains and H\"older domains as we recalled before Lemma \ref{lem-separation},
one can use quasiconformal mappings via Lemma \ref{lem-separation} to construct sufficient many examples;
see Gehring-V\"ais\"al\"a \cite{gv65}.
However, there is not a complete classification of domains under (quasi)conformal mappings
as in the two dimensional case. Indeed, it is still an important open question
whether one can characterize domains that are quasiconformally equivalent to the
unit ball in $\mathbb{R}^n$, $n\ge 3$; see Gehring \cite{g86}.

\subsection{Further questions}
\hskip\parindent In this subsection, let us discuss some remaining questions related to our results.

There is a different type of  Korn inequality
$$\lf\|D\mathbf{v}\r\|_{L^p(\Omega)}\le C\lf\{\|\epsilon(\mathbf{v})\|_{L^p(\Omega)}+\lf\|\mathbf{v}\r\|_{L^p(\Omega)}\r\}. \leqno(\widehat{K}_{p,w})$$
Obviously, $(\widehat{K}_{p,w})$ is weaker than $(\widehat{K}_p)$. Even on simply connected planar domain, we do not know how to obtain the geometric counterpart of $(\widehat{K}_{p,w})$. Indeed, as observed in \cite{jk15},
if two disjoint domains support $(\widehat{K}_{p,w})$, respectively,
then their union admits a $(\widehat{K}_{p,w})$, possibly with larger constant. However, $(K_p)$ or $(\widehat{K}_p)$
does not have this property. Therefore, in this case it looks hard to
derive the same geometric counterpart as from $(K_p)$ or $(\widehat{K}_p)$.

There is also an open problem regarding the relation of the constants in $(K_2)$ and the Babu\v ska-Aziz inequality,  which
is open even  on Lipschitz domains; see \cite{cd15}. Our result does not give any information on the optimal constant
in these inequalities.

\subsection*{Acknowledgment}
 \hskip\parindent The authors would like to thank Prof. Pekka Koskela for suggesting this topic and for valuable discussions.
R. Jiang was partially supported by National Natural Science Foundation of China (Nos. 11671039 \& 11626250). The research of A. Kauranen has been supported by the Academy of Finland via the Centre of Excellence in Analysis and Dynamics Research (project No. 271983)

\vspace{0.4cm}

\noindent Renjin Jiang$^{1,2}$\& Aapo Kauranen$^{3}$

\

\noindent
1.  Center for Applied Mathematics,  Tianjin University, Tianjin 300072, China

\noindent 2. School of Mathematical Sciences, Beijing Normal University, Beijing 100875, China

\

\noindent 3. Department of Mathematics and Statistics, University of Jyv\"{a}skyl\"{a}, P.O. Box 35 (MaD),
FI-40014, Finland

\

\noindent{\it E-mail addresses}:
\texttt{rejiang@tju.edu.cn}

\hspace{2.3cm}
\texttt{aapo.p.kauranen@jyu.fi}


\vspace{-0.3cm}
\begin{thebibliography}{999}
\bibitem{adm06} Acosta G., Dur\'an R.G., Muschietti M.A.,
Solutions of the divergence operator on John domains, Adv. Math. 206 (2006), 373-401.

\vspace{-0.3cm}
\bibitem{adl06} Acosta G., Dur\'an R.G., Lombardi A.L., Weighted
Poincar\'e and Korn inequalities for H\"older $\az$ domains, Math.
Methods Appl. Sci. 29 (2006), 387-400.

\vspace{-0.3cm}
\bibitem{adf13} Acosta G., Dur\'an R.G., L\'opez Garc\'ia F., Korn inequality
and divergence operator: Counter-examples and optimality of weighted estimates,
Proc. Amer. Math. Soc. 141 (2013), 217-232.


\vspace{-0.3cm}
\bibitem{ag85} Astala K., Gehring F.W., Quasiconformal analogues of theorems
of Koebe and Hardy-Littlewood, Michigan Math. J. 32 (1985), no. 1, 99-107.

\vspace{-0.3cm}
\bibitem{ba72} Babu\v ska I., Aziz A.K., Survey lectures on the mathematical foundations of the finite element method.
With the collaboration of G. Fix and R. B. Kellogg. The mathematical foundations of the finite element method with applications to partial differential equations (Proc. Sympos., Univ. Maryland, Baltimore, Md., 1972), pp. 1-359. Academic Press, New York, 1972.


\vspace{-0.3cm}
\bibitem{bk95} Buckley S., Koskela P., Sobolev-Poincar\'e implies
John, Math. Res. Lett. 2 (1995), 577-593.

\vspace{-0.3cm}
\bibitem{cc05}   Ciarlet P.G., Ciarlet P.Jr.,
Another approach to linearized elasticity and a new proof of Korn's inequality,
Math. Models Methods Appl. Sci. 15 (2005), 259-271.


\vspace{-0.3cm}
\bibitem{ci14}  Cianchi A., Korn type inequalities in Orlicz spaces, J. Funct. Anal. 267 (2014), 2313-2352.


\vspace{-0.3cm}
\bibitem{cfm05} Conti S., Faraco D., Maggi F., A new approach to
counterexamples to $L^1$ estimates: Korn's inequality, geometric rigidity,
and regularity for gradients of separately convex functions, Arch. Ration.
Mech. Anal. 175 (2005), 287-300.

\vspace{-0.3cm}
\bibitem{cd15}  Costabel M., Dauge M., On the inequalities of Babu\v{s}ka-Aziz, Friedrichs and Horgan-Payne,
Arch. Ration. Mech. Anal. 217 (2015), no. 3, 873-898.


\vspace{-0.3cm}
\bibitem{drs10} Diening L., Ru$\check{\rm z}$i$\check{\rm c}$ka M., Schumacher, K.,
A decomposition technique for John domains, Ann.
Acad. Sci. Fenn. Math., 35 (2010), 87-114.

\vspace{-0.3cm}
\bibitem{dl76}
Duvaut G., Lions J.-L., Inequalities in Mechanics and Physics, Springer, 1976.

\vspace{-0.3cm}
\bibitem{f37} Friedrichs K.O., On certain inequalities and characteristic value problems for analytic functions and for functions of two variables,
Trans. Amer. Math. Soc. 41 (1937),  321-364.

\vspace{-0.3cm}
\bibitem{f47} Friedrichs K.O., On the boundary-value problems of the theory of
elasticity and Korn's inequality, Ann. of Math. 48 (2) (1947) 441-471.

\vspace{-0.3cm}
\bibitem{g86}  Gehring F.W., Topics in quasiconformal mappings. Proceedings of the International Congress of Mathematicians, Vol. 1, 2 (Berkeley, Calif., 1986), 62-80.

     \vspace{-0.3cm}
\bibitem{gv65}  Gehring F.W., V\"ais\"al\"a J., The coefficients of quasiconformality of domains in space, Acta Math. 114 (1965), 1-70.


\vspace{-0.3cm}
\bibitem{hk94} Heinonen J., Koskela P., {$A_\infty$}-condition for the {J}acobian of a
              quasiconformal mapping, Proc. Amer. Math. Soc. 120 (2) (1994) 535-543.

\vspace{-0.3cm}
\bibitem{hk91} Herron D.A., Koskela P., Sobolev extension and quasiconformal circle domains, J. Anal. Math. 57 (1991), 172-202.

\vspace{-0.3cm}
\bibitem{ho95} Horgan C.O., Korn's inequalities and their applications
in continuum mechanics, SIAM Review, 37 (1995), 491-511.

\vspace{-0.3cm}
\bibitem{hp83} Horgan C.O., Payne L.E., On inequalities of Korn,
Friedrichs and  Babu\v ska-Aziz,  Arch. Rational Mech. Anal. 82 (1983), 165-179.

\vspace{-0.3cm}
\bibitem{jk15}  Jiang R.,  Kauranen A.,
Korn inequality on irregular domains, J. Math. Anal. Appl. 423 (2015), 41-59.

\vspace{-0.3cm}
\bibitem{jkk14}
 Jiang R.,  Kauranen A., Koskela P., Solvability of the divergence equation implies John
 via Poincar\'e inequality, Nonlinear Anal. 101 (2014), 80-88.

\vspace{-0.3cm}
\bibitem{j61} John F., Rotation and strain, Comm. Pure Appl. Math. 4 (1961) 391-414.

\vspace{-0.3cm}
\bibitem{jo81}
 Jones P.W., Quasiconformal mappings and extendability of functions in Sobolev spaces, Acta Math. 147 (1981), 71-88.

\vspace{-0.3cm}
\bibitem{ko89} Kondratiev V.A., Oleinik O.A., On Korn's inequalities,
C. R. Acad. Sci. Paris S\'er. I Math, 308 (1989), 483-487.

\vspace{-0.3cm}
\bibitem{kos15} Koskela P., Lectures on planar Sobolev extension domains, preprint.


\vspace{-0.3cm}
\bibitem{kos99} Koskela P., Removable sets for Sobolev spaces,  Ark. Mat. 37(2), 291-304 (1999)

\vspace{-0.3cm}
\bibitem{ms79} Martio O., Sarvas J., Injectivity theorems in plane and space,
Ann. Acad. Sci. Fenn. Ser. A I Math. 4 (1979), 383-401.



\vspace{-0.3cm}
\bibitem{mm71} Mosolov P. P., Mjasnikov V. P., A proof of {K}orn's
inequality. Dokl. Akad. Nauk SSSR 201 (1971), 36-39

\vspace{-0.3cm}
\bibitem{ne67} Ne\v cas J., Les m\'ethodes directes en th\'eorie des \'equations elliptiques.
(French) Masson et Cie (Eds.), Paris; Academia, Editeurs, Prague, 1967.

\vspace{-0.3cm}
\bibitem{nit81} Nitsche J.A., On Korn's second inequality,
RAIRO J. Numer. Anal., 15 (1981) 237-248.

\vspace{-0.3cm}
\bibitem{nv91}  N\"akki R., V\"ais\"al\"a J., John disks,
Exposition. Math. 9 (1991), 3-43.


\vspace{-0.3cm}
\bibitem{or62} Ornstein D., A non-equality for differential operators in the $L_1$ norm, Arch. Rational Mech. Anal. 11 (1962), 40-49.

\vspace{-0.3cm}
\bibitem{ru13} Russ E., A survey about the equation $\mathrm{div}{\mathbf u}=f$ in bounded domains of $\mathbb{R}^n$,
Vietnam J. Math. 41 (2013), 369-381.

\vspace{-0.3cm}
\bibitem{s70}  Stein E., Singular integrals and differentiability properties of functions,
Princeton Mathematical Series, No. 30 Princeton University Press, Princeton, N.J. (1970).

\vspace{-0.3cm}
\bibitem{ti01} Tiero A., On inequalities of Korn, Friedrichs,
Magenes-Stampacchia-Ne$\mathrm{\breve{c}}$as and Babu$\mathrm{\breve{s}}$ka-Aziz, Z. Anal. Anwendungen 20
(2001), 215-222.

\vspace{-0.3cm}
\bibitem{ti72} Ting T.W., Generalized Korn's inequalities, Tensor 25 (1972), 295-302.


\end{thebibliography}
\end{document}